\documentclass[a4paper, 10pt, notitlepage]{article}

\usepackage{amsthm, amsmath, amssymb, latexsym}
\usepackage{mathrsfs,cite}
\usepackage[multiple]{footmisc}
\usepackage{graphicx}

\theoremstyle{plain}
\newtheorem{theorem}{Theorem}
\newtheorem{lemma}{Lemma}
\newtheorem{corollary}{Corollary}
\newtheorem{proposition}{Proposition}

\theoremstyle{definition}
\newtheorem{definition}{Definition}
\newtheorem{example}{Example}

\theoremstyle{remark}
\newtheorem{remark}{Remark}

\DeclareMathOperator{\cl}{cl}
\DeclareMathOperator{\co}{co}
\DeclareMathOperator{\cone}{cone}
\DeclareMathOperator{\interior}{int}
\DeclareMathOperator{\linhull}{span}
\DeclareMathOperator{\dist}{dist}

\author{M.V. Dolgopolik\footnote{Institute for Problems in Mechanical Engineering of the Russian Academy of Sciences,
Saint Petersburg, Russia}}
\title{Nonlocal error bounds for piecewise affine functions}

\begin{document}

\maketitle

\begin{abstract}
The paper is devoted to a detailed analysis of nonlocal error bounds for nonconvex piecewise affine functions.
We both improve some existing results on error bounds for such functions and present completely new necessary and/or
sufficient conditions for a piecewise affine function to have an error bound on various types of bounded and unbounded
sets. In particular, we show that any piecewise affine function has an error bound on an arbitrary bounded set and
provide several types of easily verifiable sufficient conditions for such functions to have an error bound on unbounded
sets. We also present general necessary and sufficient conditions for a piecewise affine function to have an error bound
on a finite union of polyhedral sets (in particular, to have a global error bound), whose derivation reveals a structure
of sublevel sets and recession functions of piecewise affine functions.
\end{abstract}

\section{Introduction}

Piecewise affine functions have been an object of active research for many years. General topological and
order-theoretic properties of the set of piecewise affine (and locally piecewise affine) functions were studied in
\cite{AliprantisHarrisTourky,AddebTroitsky}. Various representations of piecewise affine functions, such as max-min,
min-max, and DC (Difference-of-Convex functions) representations, were studied in
\cite{GorokhovikZorko,Gorokhovik,Ovchinnikov}, while algorithms for constructing such representations were developed in
\cite{KripfganzSchulze,SchluterDarup,Angelov,Dolgopolik_OMS}. The surjectivity property for piecewise affine maps was
analysed in \cite{RheinboldtVandergraft,Radons}, while the bijectivity of such functions was studied in
\cite{KuhnLowen}. 

Error bound property is an important concept in variational analysis having multiple applications
\cite{Pang,FabianHenrion,Kruger,BolteNguyen,CuongKruger,Aze}. Local error bounds have attracted more attention of
researcher than global ones, since theorems on nonlocal/global error bound property for general nonlinear mappings
often involve conditions \cite{Aze,CuongKruger} that are very hard to verify in particular cases. Nonetheless, in some
specific cases (such as convex \cite{LiSinger,WangPang,Deng}, piecewise convex \cite{Li}, DC \cite{LeThiDing}, and
polynomial \cite{Vui} cases) one can use a structure of the problem to obtain simple conditions ensuring the
nonlocal/global error bound property. In the piecewise affine case, such conditions can be expressed in terms of the
so-called recession function of a piecewise affine function \cite{Gowda1996}. However, to the best of the author's
knowledge, nonlocal error bounds (that is, error bounds on various types of bounded and unbounded sets) for piecewise
affine functions, as well as conditions for such functions to have the global error bound property that do not involve
the recession function, have not been properly studied before.

The main goal of this paper is to present a detailed analysis of \textit{nonlocal} error bounds for \textit{nonconvex}
piecewise affine functions on various types of sets. We aim at improving some existing results on this topic, as well as
obtaining new necessary and/or sufficient conditions for a piecewise affine function to have a nonlocal error bound. To
this end, we heavily utilise Gorokhovik-Zorko's representation theorem for piecewise affine functions
\cite{GorokhovikZorko,Gorokhovik} that allows one to better understand a structure of sublevel sets of such functions
and conditions ensuring that they have nonlocal error bounds.

We prove that a piecewise affine function always has an error bound on an arbitrary bounded set and provide necessary
and/or sufficient conditions for such function to have an error bound on various types of unbounded sets. In the
unbounded case, we obtain simple and easily verifiable sufficient conditions for a piecewise affine function to have an
error bound, as well as more theoretical necessary and sufficient conditions that might be less appealing for
applications, but nonetheless reveal deep interrelations between some properties of piecewise affine functions, their
sublevel sets, recession functions, and ``flat parts'' (that is, polyhedral sets on which a piecewise affine function is
constant). We also apply all these results to an analysis of error bounds for systems of piecewise affine equality and
inequality constraints.

It should be mentioned that apart from many completely new results, we present improved versions of Robinson's
\cite{Robinson} and Gowda's \cite{Gowda1996} results on error bounds for piecewise affine functions. Robinson
\cite{Robinson} showed that for a piecewise affine function $F \colon \mathbb{R}^d \to \mathbb{R}^m$ 
\textit{there exists} $\rho > 0$ such that the function $\| F(\cdot) \|$ has an error bound on the set 
$V(\rho) = \{ x \mid \| F(x) \| \le \rho \}$. We improve this result by showing how $\rho > 0$ from the Robinson's
theorem can be \textit{easily estimated} (see Theorem~\ref{thrm:Robinson}) and how the procedure for estimating $\rho$
can in some cases be used to verify that the function $\| F(\cdot) \|$ has a global error bound. 

In turn, Gowda in \cite{Gowda1996} presented necessary and sufficient conditions for a piecewise affine function to
have a global error bound in terms of the recession function. We extend Gowda's result to the case of error bounds on a
finite union of polyhedral sets and, moreover, obtain new necessary and sufficient conditions for a piecewise affine
function to have an error bound on a finite union of polyhedral sets (in particular, a global error bound) that are not
based on the use of the recession function (Lemma~\ref{lem:ErrorBound_StratifiedReccCones}).

The paper is organised as follows. Some auxiliary definitions and results from convex analysis and related fields that
are used throughout the article are collected in Section~\ref{sect:Preliminaries}, while some auxiliary properties of
piecewise affine functions are studied in Section~\ref{sect:PiecewiseAffine}. Section~\ref{sect:ErrorBoundsPAFunc} is
devoted to error bounds for real-valued nonconvex piecewise affine functions. Subsection~\ref{subsect:EB_BoundedSets}
contains an improved version of the Robinson's result \cite{Robinson} and an analysis of error bounds for piecewise
affine functions on bounded sets. Several types of sufficient conditions for the existence of an error bound on
unbounded sets are given in Subsection~\ref{subsect:EB_UnboundedSets}, while general necessary and sufficient conditions
for a piecewise affine function to have an error bound on a finite union of polyhedral sets are studied in
Subsection~\ref{subsect:EB_GeneralCond}. Finally, in Section~\ref{sect:PAConstraints} the results on error bounds for
piecewise affine functions are applied to an analysis of error bounds for systems of piecewise affine equality and
inequality constraints.

\section{Preliminaries}
\label{sect:Preliminaries}

Let us recall some auxiliary definitions and results that will be used throughout the article. First, we present 
a particular version of the well-known Hoffman's theorem \cite{Hoffman} on the error bounds for
systems of linear inequalities. For any set $Q \subset \mathbb{R}^d$ and $x \in \mathbb{R}^d$ denote by
$\dist(x, Q) = \inf_{y \in Q} \| x - y \|$ the distance between $x$ and $Q$, where $\| \cdot \|$ is a norm on
$\mathbb{R}^d$.

\begin{theorem}[Hoffman]
Let $A x \le b$ with $A \in \mathbb{R}^{m \times d}$ and $b = (b_1, \ldots, b_m) \in \mathbb{R}^m$ be a consistent
system of linear inequalities. Then there exists $\tau > 0$ such that
\[
  \tau \dist(x, \Omega) \le \max_{i \in \{ 1, \ldots, m \}} \max\{ 0, \langle A_i, x \rangle - b_i \},
\]
where $\Omega = \{ x \in \mathbb{R}^d \mid A x \le b \}$, $A_i$ are the rows of the matrix $A$, and 
$\langle \cdot, \cdot \rangle$ is the inner product in $\mathbb{R}^d$.
\end{theorem}

Let $Q \subseteq \mathbb{R}^d$ be a nonempty set. \textit{The recession cone} $0^+ Q$ of the set $Q$
consists of all those vectors $z \in \mathbb{R}^d$ for which one can find $x \in Q$ such that $x + \lambda z \in Q$ for
all $\lambda \ge 0$. The set $Q$ is called a \textit{polytope}, if it is the convex hull of a finite number of points.
The set $Q$ is called \textit{polyhedral}, if it is the intersection of a finite number of closed half-spaces. 
The Motzkin theorem (see \cite[Thms.~19.1 and 19.5]{Rockafellar}) provides a useful representation of polyhedral sets.

\begin{theorem}[Motzkin]
A set $Q \subseteq \mathbb{R}^d$ is polyhedral if and only if there exists a polytope $P \subset \mathbb{R}^d$ and a
polyhedral cone $K \subset \mathbb{R}^d$ such that $Q = P + K$. Moreover, such cone $K$ is uniquely defined and equal to
the recession cone $0^+ Q$.
\end{theorem}

Let us now recall the definition of piecewise affine function \cite{KripfganzSchulze,GorokhovikZorko,Gorokhovik}.

\begin{definition} \label{def:PiecewiseAffineFunction}
A finite family $\sigma = \{ Q_1, \ldots, Q_s \}$ of polyhedral subsets $Q_i$ of $\mathbb{R}^d$ is called 
\textit{a polyhedral partition} of $\mathbb{R}^d$, if 
\[
  \bigcup_{i = 1}^s Q_i = \mathbb{R}^d, \quad \interior Q_i \ne \emptyset, \quad
  \interior Q_i \cap \interior Q_j = \emptyset \quad
  \forall i, j \in \{ 1, \ldots, s \}, \: i \ne j,
\]
where $\interior Q$ is the topological interior of a set $Q \subset \mathbb{R}^d$. A function 
$F \colon \mathbb{R}^d \to \mathbb{R}^m$ is called \textit{piecewise affine}, if there exists a polyhedral partition
$\sigma = \{ Q_1, \ldots, Q_s \}$ of $\mathbb{R}^d$ and a collection of affine functions
$F_i \colon \mathbb{R}^d \to \mathbb{R}^m$, $F_i(x) = a_i + V_i x$ with $V_i \in \mathbb{R}^{m \times d}$ and 
$a_i \in \mathbb{R}^m$, such that $F(x) = F_i(x)$ for all $x \in Q_i$ and $i \in \{ 1, \ldots, s \}$.
\end{definition}

\begin{remark}
As was noted in \cite{Gorokhovik}, the assumption that the sets $Q_i$ from the polyhedral partition have nonempty
interiors is, in fact, redundant. It is sufficient to suppose that only the relative interiors of the sets $Q_i$, 
$i \in \{ 1, \ldots, s \}$, are pairwise disjoint.
\end{remark}

It is worth mentioning that the set of all piecewise affine functions from $\mathbb{R}^d$ to $\mathbb{R}^m$ is closed
under addition, multiplication by scalar, as well as coordinate-wise supremum and infimum of finite families of
functions. Furthermore, this set is the smallest vector lattice (with respect to pointwise operations) containing all
affine functions. Finally, the composition of piecewise affine functions is also a piecewise affine function
\cite{Gorokhovik}.

Apart from representations of piecewise affine functions in terms of polyhedral partitions, one often has to deal with
various analytical representations of such functions. As was proved in \cite{GorokhovikZorko,Gorokhovik}, among various
analytical representations of piecewise affine functions there always exist a natural DC (Difference-of-Convex
functions) decomposition of such functions and a min-max/max-min representation that are especially convenient for
theoretical analysis.

\begin{theorem}[Gorokhovik-Zorko] \label{thrm:DCRepresentation}
Let $F \colon \mathbb{R}^d \to \mathbb{R}^m$ be a given function. The following assertions are equivalent:
\begin{enumerate}
\item{$F$ is piecewise affine;}

\item{$F$ can be represented in the form 
\[
  F(x) = \sup_{i \in I} F_i(x) + \inf_{j \in J} G_j(x) \quad \forall x \in \mathbb{R}^d
\]
for some finite families of affine functions $F_i \colon \mathbb{R}^d \to \mathbb{R}^m$, 
$i \in I := \{ 1, \ldots, \ell \}$, and $G_j \colon \mathbb{R}^d \to \mathbb{R}^m$, $j \in J := \{ 1, \ldots, s \}$,
where the supremum and the infimum are taken with respect to the coordinate-wise partial order in $\mathbb{R}^m$;
}

\item{$F$ can be represented in the form
\[
  F(x) = \inf_{i \in I} \sup_{j \in J(i)} F_{ij}(x) \quad \forall x \in \mathbb{R}^d
\]
for some affine functions $F_{ij} \colon \mathbb{R}^d \to \mathbb{R}^m$, $i \in I := \{ 1, \ldots, \ell \}$, and
$j \in J(i) = \{ 1, \ldots, s(i) \}$;
}

\item{$F$ can be represented in the form
\[
  F(x) = \sup_{i \in I} \inf_{j \in J(i)} F_{ij}(x) \quad \forall x \in \mathbb{R}^d
\]
for some affine functions $F_{ij} \colon \mathbb{R}^d \to \mathbb{R}^m$, $i \in I := \{ 1, \ldots, \ell \}$, and
$j \in J(i) = \{ 1, \ldots, s(i) \}$. 
}
\end{enumerate}
\end{theorem}

\begin{remark}
Methods for constructing analytical representations of piecewise affine functions from their representations via
polyhedral partitions were studied in \cite{KripfganzSchulze,SchluterDarup}. In turn, methods for constructing DC
decompositions of piecewise affine functions (which can be used to construct max-min and min-max representations of
such functions) from their arbitrary analytical representations were developed in
\cite{Angelov} (see also \cite{Dolgopolik_OMS}).
\end{remark}

\section{Min-max representation and lower $0$-level set}
\label{sect:PiecewiseAffine}

Let us prove some useful auxiliary results on piecewise affine functions. Hereinafter, let 
$f \colon \mathbb{R}^d \to \mathbb{R}$ be a piecewise affine function. By Theorem~\ref{thrm:DCRepresentation} the
function $f$ can be represented in the min-max form
\begin{equation} \label{eq:MinMaxRepresentation}
  f(x) = \min_{i \in I} \max_{j \in J(i)} (a_{ij} + \langle v_{ij}, x \rangle) \quad \forall x \in \mathbb{R}^d
\end{equation}
for some $a_{ij} \in \mathbb{R}$, $v_{ij} \in \mathbb{R}^d$, $i \in I = \{ 1, \ldots, \ell \}$, and 
$j \in J(i) = \{ 1, \ldots, s(i) \}$. For all $i \in I$ denote
\begin{equation} \label{eq:fi_def}
  f_i(x) = \max_{j \in J(i)} (a_{ij} + \langle v_{ij}, x \rangle), \quad
  f_i^* = \inf_{x \in \mathbb{R}^d} f_i(x).
\end{equation}
By definition $f = \min_{i \in I} f_i$. Introduce the index set $I_0 = \{ i \in I \mid f_i^* \le 0 \}$. 

For any function $g \colon \mathbb{R}^d \to \mathbb{R}$ denote by $S(g) = \{ x \in \mathbb{R}^d \mid g(x) \le 0 \}$ the
lower $0$-level set of $g$, and let $[g]_+(x) = \max\{ g(x), 0 \}$. With the use of min-max representation
\eqref{eq:MinMaxRepresentation} we can easily describe the set $S(f)$ and the function $[f]_+$ in terms of the convex
functions $f_i$.

\begin{lemma} \label{lem:MinMax}
The following statements hold true:
\begin{enumerate}
\item{$S(f) = \bigcup_{i \in I_0} S(f_i)$; \label{st:MinMax_1}}

\item{$0^+ S(f) = \bigcup_{i \in I_0} 0^+ S(f_i)$; \label{st:MinMax_2}}

\item{$\dist(x, S(f)) = \min_{i \in I_0} \dist(x, S(f_i))$ for any $x \in \mathbb{R}^d$; \label{st:MinMax_3}}

\item{$[f]_+(x) = \min_{i \in I} [f_i]_+(x)$ for any $x \in \mathbb{R}^d$. \label{st:MinMax_4}}
\end{enumerate}
\end{lemma}

\begin{proof}
\ref{st:MinMax_1}. From the equality $f = \min_{i \in I} f_i$ it obviously follows that
$S(f) = \bigcup_{i \in I} S(f_i)$. Hence taking into account the fact that for any $i \in I \setminus I_0$ the set
$S(f_i)$ is obviously empty one obtains  the required result.

\ref{st:MinMax_2}. The validity of this equality follows directly from the first statement of the lemma and 
the equality $0^+ \bigcup_{i \in I_0} S(f_i) = \bigcup_{i \in I_0} 0^+ S(f_i)$ that can be readily verified directly
with the use of the convexity of the sets $S(f_i)$.

\ref{st:MinMax_3}. With the use of the first statement of the lemma one gets that
\begin{align*}
  \dist(x, S(f)) &= \inf_{y \in S(f)} \| x - y \| = \inf_{y \in \bigcup_{i \in I_0} S(f_i)} \| x - y \|
  \\
  &= \min_{i \in I_0} \inf_{y \in S(f_i)} \| x - y \| = \min_{i \in I_0} \dist(x, S(f_i))
\end{align*}
for any $x \in \mathbb{R}^d$.

\ref{st:MinMax_4}. Fix any $x \in \mathbb{R}^d$. For all $i \in I$ one has $f(x) \le f_i(x)$, which implies that
$[f]_+(x) \le [f_i]_+(x)$ and inequality $[f]_+(x) \le \min_{i \in I} [f_i]_+(x)$ holds true. To prove the converse
inequality, note that by definition there exists $i \in I$ such that 
\[
  [f]_+(x) := \max\Big\{ 0, \min_{i \in I} f_i(x) \Big\} = \max\{ 0, f_i(x) \} = [f_i]_+(x). 
\]
Therefore $[f]_+(x) \ge \min_{i \in I} [f_i]_+(x)$.
\end{proof}

Recall that \textit{the recession function} $f^{\infty}$ of $f$ is defined as
\[
  f^{\infty}(x) = \lim_{\lambda \to + \infty} \frac{f(\lambda x)}{\lambda}
\]
(see \cite{Gowda1996}). As is easily seen,
\begin{equation} \label{eq:RecFuncViaMinMaxRepresentation}
  f^{\infty}(x) = \min_{i \in I} \max_{j \in J(i)} \langle v_{ij}, x \rangle \quad \forall x \in \mathbb{R}^d,
\end{equation}
that is, $f^{\infty}$ is a positively homogeneous piecewise affine function. Let us show how the $0$-sublevel set
$S([f]_+^{\infty})$ of the recession function can be described in terms of the functions $f$ and $f_i$.

\begin{lemma} \label{lem:RecFuncSublevel}
The following equality holds true:
\[
  S([f]_+^{\infty}) = 0^+ S(f) \cup \Big( \bigcup_{i \in I \colon f_i^* > 0} 0^+ S(f_i - f_i^*) \Big).
\]
\end{lemma}

\begin{proof}
From \eqref{eq:RecFuncViaMinMaxRepresentation} it follows that
\begin{equation} \label{eq:RecFunc_SublevelSet_Expl}
  S([f]_+^{\infty}) = \bigcup_{i \in I} 
  \Big\{ x \in \mathbb{R}^d \Bigm\vert \langle v_{ij}, x \rangle \le 0 \enspace \forall j \in J(i) \Big\}.
\end{equation}
Bearing in mind the definition of $f_i$ (see \eqref{eq:fi_def}) one obtains
\begin{align} \label{eq:ConvexFuncSublevelSet}
  S(f_i) &= \Big\{ x \in \mathbb{R}^d \Bigm\vert 
  a_{ij} + \langle v_{ij}, x \rangle \le 0 \enspace \forall j \in J(i) \Big\} \quad \forall i \in I_0,
  \\ \notag
  S(f_i - f_i^*) &= \Big\{ x \in \mathbb{R}^d \Bigm\vert 
  a_{ij} + \langle v_{ij}, x \rangle \le f_i^* \enspace \forall j \in J(i) \Big\} \quad \forall i \in I \setminus I_0,
\end{align}
which obviously implies that
\begin{align*}
  0^+ S(f_i) &= \Big\{ x \in \mathbb{R}^d \Bigm\vert
  \langle v_{ij}, x \rangle \le 0 \enspace \forall j \in J(i) \Big\} \quad \forall i \in I_0,
  \\
  0^+ S(f_i - f_i^*) &= \Big\{ x \in \mathbb{R}^d \Bigm\vert 
  \langle v_{ij}, x \rangle \le 0 \enspace \forall j \in J(i) \Big\} \quad \forall i \in I \setminus I_0.
\end{align*}
Combining \eqref{eq:RecFunc_SublevelSet_Expl}, the equalities above, and the second statement of Lemma~\ref{lem:MinMax}
we arrive at the required result.
\end{proof}

\section{Nonlocal error bounds for piecewise affine functions}
\label{sect:ErrorBoundsPAFunc}

In this section we study error bounds for piecewise affine functions on various types of sets. Our main goal is to show
that nonconvex piecewise affine functions have an error bound on any bounded set and provide necessary and/or sufficient
conditions for such functions to have an error bound on an unbounded set.

\subsection{Error bounds on bounded sets}
\label{subsect:EB_BoundedSets}

Recall that $f$ is said to have \textit{an error bound} with constant $\tau > 0$ on a set $V \subset \mathbb{R}^d$, if
\begin{equation} \label{def:ErrorBound}
  \tau \dist(x, S(f)) \le [f]_+(x) \quad \forall x \in V.
\end{equation}
The supremum of all those $\tau$ for which inequality \eqref{def:ErrorBound} holds true is denoted by $\tau(f, V)$ or
simply $\tau(V)$, if the function $f$ is fixed. Finally, $f$ is said to have a global error bound, if there exists 
$\tau > 0$ such that inequality \eqref{def:ErrorBound} holds true for $V = \mathbb{R}^d$.

We start our analysis of error bounds for piecewise affine functions by proving a new improved version of the Robinson's
result on error bounds for piecewise affine functions from \cite{Robinson}. 

\begin{theorem} \label{thrm:Robinson}
Let $S(f)$ be nonempty. Then the following statements hold true:
\begin{enumerate}
\item{if $f_i^* \le 0$ for all $i \in I$, then $f$ has a global error bound;}

\item{if there exists $i \in I$ such that $f_i^* > 0$, then $f$ has an error bound on the set 
$V = \{ x \in \mathbb{R}^d \mid f(x) < \rho \}$ with $\rho = \min\{ f_i^* \mid i \in I \colon f_i^* > 0 \}$.}
\end{enumerate}
\end{theorem}

\begin{proof}
\textbf{Case I.} Suppose that $f_i^* \le 0$ for all $i \in I$. Note that for each $i \in I$ the set $S(f_i)$ is nonempty
and $f_i$ has a global error bound with some constant $\tau_i > 0$. 

Indeed, if $f_i^* < 0$, then the set $S(f_i)$ is obviously nonempty. If $f_i^* = 0$, then taking into account the fact
that a bounded below piecewise affine function attains a global minimum (see, e.g. \cite[Thm.~4.7]{Dolgopolik_OMS}) one
can conclude that the set $S(f_i)$ is nonempty, since it contains a global minimizer of $f_i$.

As was noted above, the set $S(f_i)$ coincides with the set of solutions of the corresponding systems of linear
inequalities (see \eqref{eq:ConvexFuncSublevelSet}). Therefore, by Hoffman's theorem there exists $\tau_i > 0$ such that
\begin{align*}
  \tau_i \dist(x, S(f_i)) &\le \max_{j \in J(i)} \max\{ 0, a_{ij} + \langle v_{ij}, x \rangle \} 
  \\
  &= \max\Big\{ 0, \max_{j \in J(i)} (a_{ij} + \langle v_{ij}, x \rangle) \Big\} = [f_i(x)]_+
\end{align*}
for all $x \in \mathbb{R}^d$, that is, $f_i$ has a global error bound with constant $\tau_i$.

Now, applying Lemma~\ref{lem:MinMax} one gets that
\begin{align*}
  [f]_+(x) = \min_{i \in I} [f_i]_+(x) &\ge \min_{i \in I} \tau_i \dist(x, S(f_i)) 
  \\
  &\ge \tau \min_{i \in I} \dist(x, S(f_i)) = \tau \dist(x, S(f)), 
\end{align*}
for any $x \in \mathbb{R}^d$, where $\tau = \min_{i \in I} \tau_i > 0$. In other words, $f$ has a global error bound
with constant $\tau$.

\textbf{Case II.} Suppose now that $f_i^* > 0$ for some $i \in I$. Introduce the function 
$g(x) = \min_{i \in I_0} f_i(x)$. Recall that $f = \min_{i \in I} f_i$ and $f_i(x) \ge f_i^* \ge \rho$ for any 
$x \in \mathbb{R}^d$ and $i \notin I_0$ by definitions. Therefore, if $f(x) < \rho$ for some $x \in \mathbb{R}^d$, then
there exists $i \in I_0$ such that $f(x) = f_i(x)$. Consequently, $f(x) = g(x)$ for any $x \in V$. Moreover,
$S(f) = S(g)$. 

Indeed, the validity of the inclusion $S(g) \subseteq S(f)$ follows from the fact that $g(x) \ge f(x)$ for any 
$x \in \mathbb{R}^d$ by the definition of $g$. In turn, if $x \in S(f)$, then $f(x) \le 0 < \rho$, which implies that
$g(x) = f(x) \le 0$ and $x \in S(g)$, that is, $S(f) \subseteq S(g)$.

By the first part of the proof $g$ has a global error bound. Hence, as one can readily check, $f$ has an error bound on
$V$, since $f$ coincides with $g$ on this set and $S(f) = S(g)$.
\end{proof}

As a simple corollary to the theorem above we can prove that positively homogeneous piecewise affine functions (such
functions are sometimes called \textit{piecewise linear} \cite{GorokhovikZorko,Gorokhovik}) always have a global error
bound.

\begin{corollary}
Let $f$ be positively homogeneous. Then it has a global error bound.
\end{corollary} 

\begin{proof}
Since $f$ is positively homogeneous, by \cite[Thm.~3.2]{GorokhovikZorko} this function can be represented in the form
\eqref{eq:MinMaxRepresentation} with $a_{ij} = 0$ for all $i$ and $j$. Therefore, for the functions $f_i$ (see
\eqref{eq:fi_def}) one has $f_i(0) = 0$, which implies that $f_i^* \le 0$ for all $i \in I$. Consequently, $f$
has a global error bound by Theorem~\ref{thrm:Robinson}.
\end{proof}

The two following simple examples demonstrate that in the case of non-positively homogeneous piecewise affine functions
the value $\rho$ from Theorem~\ref{thrm:Robinson} cannot be improved, but, at the same time, this theorem does not
describe the largest set on which a piecewise affine function has an error bound.

\begin{example} \label{example:Robinson}
Let $d = 1$ and
\[
  f(x) = \begin{cases}
    0, & \text{if } x \le 0, \\
    x, & \text{if } x \in [0, 1], \\
    1, & \text{if } x \ge 1. 
  \end{cases}
\]
Clearly, one has $S(f) = (- \infty, 0]$ and
\[
  f(x) = \min\big\{ f_1(x), f_2(x) \big\}, \quad f_1(x) = 1, \quad f_2(x) = \max\{ 0, x \}.
\]
Therefore $f_1^* = 1$, $f_2^* = 0$, and by Theorem~\ref{thrm:Robinson} $f$ has an error bound on the set 
$\{ x \in \mathbb{R} \mid f(x) < 1 \}$. Furthermore, as is easily seen, $f$ does not have an error bound on the set 
$\{ x \in \mathbb{R} \mid f(x) < \rho \}$ for any $\rho > 1$.
\end{example}

\begin{example}
Let $d = 1$ and
\[
  f(x) = \begin{cases}
    0, & \text{if } x \le 0, \\
    x, & \text{if } x \in [0, 1], \\
    1, & \text{if } x \in [1, 2], \\
    x - 1, & \text{if } x \ge 2. 
  \end{cases}
\]
Then one has
\[
  f(x) = \min\{ f_1(x), f_2(x) \}, \quad f_1(x) = \max\{ 0, x \}, \quad f_2(x) = \max\{ 1, x - 1 \}.
\]
Consequently, $f_1^* = 0$, $f_2^* = 1$, and by Theorem~\ref{thrm:Robinson} $f$ has an error bound on the set 
$\{ x \in \mathbb{R} \mid f(x) < 1 \}$. However, in actuality, $f$ has a global error bound and 
$\tau(\mathbb{R}) = 0.5$.
\end{example}

Denote by $U(x, r) = \{ y \in \mathbb{R}^d \mid \| x - y \| < r \}$ the open ball with centre $x$ and radius $r > 0$.
Recall that $f$ is said to have an error bound at a point $\overline{x} \in S(f)$, if there exist $\tau > 0$ and a
neighbourhood $V$ of $\overline{x}$ for which inequality \eqref{def:ErrorBound} holds true. It follows from
Theorem~\ref{thrm:Robinson} that any piecewise affine function $f$ has, in some sense, a \textit{uniform} error bound at
every $\overline{x} \in S(f)$. 

\begin{proposition} \label{prp:UniformLocalErrorBound}
For any $\overline{x} \in S(f)$ the function $f$ has an error bound at $\overline{x}$. Furthermore, there exist $r > 0$
and $\tau_0 > 0$ such that $\tau(U(\overline{x}, r)) \ge \tau_0$ for all $\overline{x} \in S(f)$, provided the set
$S(f)$ is nonempty.
\end{proposition}

\begin{proof}
If $f_i^* \le 0$ for all $i \in I$, then by Theorem~\ref{thrm:Robinson} one can set $\tau_0 = \tau(\mathbb{R}^d)$ and
choose any $r > 0$.

Therefore, suppose that $f_i^* > 0$ for some $i \in I$. Let $\rho = \min_{i \in I \setminus I_0} f_i^*$ and
denote $V = \{ x \in \mathbb{R}^d \mid f(x) < \rho \}$. Then $\tau(V) > 0$ by Theorem~\ref{thrm:Robinson}. 

From the representation \eqref{eq:MinMaxRepresentation} it obviously follows that $f$ is globally Lipschitz continuous
with Lipschitz constant 
\[
  L = \max_{i \in I, j \in J(i)} \| v_{ij} \|_*, \quad
  \| v_{ij} \|_* = \max\{ \langle v_{ij}, y \rangle \mid y \in U(0, 1) \}
\] 
Consequently, for any $\overline{x} \in S(f)$ and $x \in U(\overline{x}, r)$ with $r = \rho / L$ one has
\[
  f(x) = ( f(x) - f(\overline{x}) ) + f(\overline{x}) \le f(x) - f(\overline{x}) 
  \le L \| x - \overline{x} \| < \rho,
\]
that is, $U(\overline{x}, r) \subseteq V$. Hence for any such $\overline{x}$ the inequality 
$\tau(U(\overline{x}, r)) \ge \tau(V)$ holds true. Thus, one can set $r = \rho / L$ and $\tau_0 = \tau(V)$.
\end{proof}

With the use of the previous proposition one can easily prove that a piecewise affine function has an error bound on any
bounded set.

\begin{theorem} \label{thrm:ErrorBound_Boundedset}
Let $S(f)$ be nonempty. Then $f$ has an error bound on every bounded subset of $\mathbb{R}^d$.
\end{theorem}

\begin{proof}
Fix any bounded set $V \subset \mathbb{R}^d$. Replacing $V$ with its closure, if necessary, one can suppose that $V$ is
compact. We need to check that $f$ has an error bound on $V$. 

Let $r > 0$ and $\tau_0 > 0$ be from Proposition~\ref{prp:UniformLocalErrorBound}. Note that the set $S(f)$ is closed,
since piecewise affine functions are continuous. Therefore, the set $S(f) \cap V$ is compact, and one can find points
$x_1, \ldots, x_N$ from this set such that
\[
  S(f) \cap V \subset U := \bigcup_{k = 1}^N U(x_k, r).
\]
Denote $W = V \setminus U$. If $W$ is empty, then $V \subset U$, which, as is easily seen, implies that $f$ has an error
bound on $V$ with $\tau(V) \ge \tau_0$. Therefore, suppose that $W$ is nonempty.

The set $W$ is obviously compact. Therefore, the following values are correctly defined and finite:
\[
  \alpha = \max_{x \in W} \dist(x, S(f)), \quad \beta = \min_{x \in W} f(x).
\] 
By definition the sets $S(f)$ and $W$ do not intersect, which implies that $\alpha > 0$ and $\beta > 0$. Consequently,
one has
\[
  \frac{\beta}{\alpha} \dist(x, S(f)) \le \beta \le f(x) = [f]_+(x) \quad \forall x \in W.
\]
On the other hand, for any $x \in V \cap U$ the point $x$ belongs to some $U(x_i, r)$, which yields 
$\tau_0 \dist(x, S(f)) \le [f]_+(x)$. Therefore $f$ has an error bound on $V$ with constant 
$\tau = \min\{ \beta/\alpha, \tau_0 \} > 0$.
\end{proof}

\begin{remark} \label{rmrk:ErrorBound_BoundedSet}
From the theorem above it follows that $f$ has an error bound on any set $V \subset \mathbb{R}^d$ such
that the set $V \setminus S(f)$ is bounded, even if both $S(f)$ and $V$ are unbounded (one simply has to apply 
the theorem to the set $V \setminus S(f)$).
\end{remark}

\subsection{Error bounds on unbounded sets: sufficient conditions}
\label{subsect:EB_UnboundedSets}

Now we turn to analysis of error bounds on unbounded sets. As Example~\ref{example:Robinson} shows, a piecewise
affine function might not have a global error bound. Therefore, below we study error bounds on an arbitrary
unbounded set. First, we provide verifiable sufficient conditions for a piecewise affine function to have an error bound
on an unbounded set that do not require any information about the set $S(f)$ and show when these conditions become
necessary. 

For any nonempty set $V \subseteq \mathbb{R}^d$ denote by $\cone V$ the conic hull of $V$ (i.e. the smallest cone
containing $V$), and by $\cl V$ the closure of $V$. 

\begin{theorem} \label{thrm:ErrorBound_GrowthAtInfinity}
Let $S(f)$ be nonempty and $V \subseteq \mathbb{R}^d$ be a given set. Then for $f$ to have an error bound on $V$ it
is sufficient that
\begin{equation} \label{eq:GrowthAtInfinity}
  \liminf_{\| x \| \to + \infty, x \in V} \frac{f(x)}{\| x \|} > 0.
\end{equation}
Moreover, this condition becomes necessary, when 
\begin{equation} \label{eq:NonParallelRays}
  0^+ S(f) \cap \cl \cone V = \{ 0 \}. 
\end{equation}
\end{theorem}

\begin{proof}
\textbf{Sufficiency.} Suppose that condition \eqref{eq:GrowthAtInfinity} holds true. Then there exist $r > 0$ and
$a > 0$ such that
\[
  f(x) \ge a \| x \| \quad \forall x \in  V \setminus U(0, r).
\]
Fix any $x_0 \in S(f)$. With the use of the inequality above one obtains that for any $x \in V \setminus U(0, R)$ with
$R = \max\{ r, \| x_0 \| \}$ the following inequalities hold true:
\[
  \dist(x, S(f)) \le \| x - x_0 \| \le \| x_0 \| + \| x \| \le 2 \| x \| \le \frac{2}{a} f(x).
\]
On the other hand, by Theorem~\ref{thrm:ErrorBound_Boundedset} the function $f$ has an error bound on $U(0, R)$.
Hence $f$ has an error bound on $V$ with constant $\tau = \max\{ a/2, \tau(U(0, R)) \}$.

\textbf{Necessity.} Suppose that $f$ has an error bound on $V$ with constant $\tau > 0$ and condition
\eqref{eq:NonParallelRays} holds true. Our aim is to show that under this assumption
\[
  \liminf_{\| x \| \to +\infty, x \in V} \frac{\dist(x, S(f))}{\| x \|} > 0.
\]
Then taking into account the fact that $[f(x)]_+ \ge \tau \dist(x, S(f))$ for any $x \in V$ by the definition of error
bound one obtains the required result.

By Lemma~\ref{lem:MinMax} one has $\dist(\cdot, S(f)) = \min_{i \in I_0} \dist(\cdot, S(f_i))$. Therefore, it is
sufficient to prove that
\begin{equation} \label{eq:DistFunctGrowthAtInf_Conv}
  \liminf_{\| x \| \to +\infty, x \in V} \frac{\dist(x, S(f_i))}{\| x \|} > 0 \quad \forall i \in I_0.
\end{equation}
Let us prove this inequality.

Fix any $i \in I_0$.  By the Motzkin theorem there exists a polytope $P$ such that $S(f_i) = P + 0^+ S(f_i)$, since
$S(f_i)$ is a polyhedral set. Therefore by the reverse triangle inequality for any $x \in \mathbb{R}^d$ one has
\begin{equation} \label{eq:DistPolyhed_via_RecCone}
\begin{split}
  \dist(x, S(f_i)) &= \inf_{y \in P, z \in 0^+ S(f_i)} \| x - y - z \| 
  \\
  &\ge \inf_{z \in 0^+ S(f_i)} \| x - z \| - C
  = \dist(x, 0^+ S(f_i)) - C
\end{split}
\end{equation}
where $C = \sup_{y \in P} \| y \|$. By Lemma~\ref{lem:MinMax} one has $0^+ S(f_i) \subseteq 0^+ S(f)$, which thanks to
our assumption implies that $0^+ S(f_i) \cap \cl \cone V = \{ 0 \}$. Consequently, one has
\[
  \beta := \min\Bigl\{ \dist(x, 0^+ S(f_i)) \Bigm\vert x \in \cl \cone V, \| x \| = 1  \Bigr\} > 0.
\]
Since the recession cone $0^+ S(f_i)$ is a cone, the function $x \mapsto \dist(x, 0^+ S(f_i))$ is positively
homogeneous. Hence
\[
  \dist(x, 0^+ S(f_i)) \ge \beta \| x \| \quad \forall x \in \cl \cone V.
\]
Combining this inequality with inequality \eqref{eq:DistPolyhed_via_RecCone} one finally obtains that
\[
  \dist(x, S(f_i)) \ge \beta \| x \| - C \quad \forall x \in \cl \cone V,
\] 
which obviously implies that condition \eqref{eq:DistFunctGrowthAtInf_Conv} holds true.
\end{proof}

\begin{remark}
It is worth mentioning that in the case when the set $V$ is a closed cone, condition \eqref{eq:NonParallelRays} simply
means that there are no rays in $V$ that are parallel to a ray contained in $S(f)$. As
Example~\ref{ex:GlobalErrorBound_FlatPieces} shows, in the case when the set $V$ is, in some sense, parallel to the set
$S(f)$, the function $f$ might have an error bound on $V$, but not satisfy condition \eqref{eq:NonParallelRays}.
\end{remark}

In the case when $V$ is a cone, one can provide somewhat less restrictive conditions for $f$ to have an error bound on
$V$ than in the theorem above. Recall that $f$ is said to be coercive on an unbounded set $V$, if $f(x_n) \to + \infty$
for any sequence $\{ x_n \} \subset V$ such that $\| x_n \| \to + \infty$ as $n \to \infty$. In the case when 
$V = \mathbb{R}^d$ we simply say that $f$ is coercive.

\begin{theorem} \label{thrm:ErrorBound_on_Cone}
Let $S(f)$ be nonempty and $V \subseteq \mathbb{R}^d$ be a cone. Then for $f$ to have an error bound on $V$ it is
sufficient that $f$ is coercive on $V$. Furthermore, this condition becomes necessary, when 
$0^+ S(f) \cap \cl V = \{ 0 \}$.
\end{theorem}

\begin{proof}
\textbf{Necessity.} Let $f$ have an error bound on $V$ and $0^+ S(f) \cap \cl V = \{ 0 \}$. Then by
Theorem~\ref{thrm:ErrorBound_GrowthAtInfinity} inequality \eqref{eq:GrowthAtInfinity} holds true, which obviously
implies that $f$ is coercive on $V$.

\textbf{Sufficiency.} Suppose now that $f$ is coercive on $V$. As was shown in the proof of Theorem~\ref{thrm:Robinson},
each function $f_i$, $i \in I_0$, has a global error bound with some constant $\tau_i$. Let us check that for each 
$i \in I \setminus I_0$ there exists $\tau_i > 0$ such that
\begin{equation} \label{eq:PositiveComponentsLowerEstimate}
  \tau_i \dist(x, S(f)) \le [f_i(x)]_+ \quad \forall x \in V.
\end{equation}
Then with the use of Lemma~\ref{lem:MinMax} one obtains
\begin{align*}
  [f(x)]_+ = \min_{i \in I} [f_i(x)]_+ 
  &\ge \min\Big\{ \min_{i \in I_0} \tau_i \dist(x, S(f_i)), \min_{i \in I \setminus I_0} \tau_i \dist(x, S(f)) \Big\}
  \\
  &\ge \tau \dist(x, S(f))
\end{align*}
for $\tau = \min_{i \in I} \tau_i$ and all $x \in V$, that is, $f$ has an error bound on $V$.

Thus, it remains to prove inequality \eqref{eq:PositiveComponentsLowerEstimate}. Fix any $i \in I \setminus I_0$. Recall
that by the definition of $I_0$ one has $f_i^* > 0$, which implies that $f_i(x) > 0$, i.e. $f_i(x) = [f_i]_+(x)$, for
all $x \in \mathbb{R}^d$.

From equality $f = \min_{i \in I} f_i$ and our assumption on coercivity of $f$ on $V$ it follows that $f_i$ is coercive
on $V$ as well. Consequently, the function $g_i(x) = \max_{j \in J(i)} \langle v_{ij}, x \rangle$ is also coercive on
$V$, since 
\[
  g_i(x) \ge f_i(x) - a_0 \quad \forall x \in \mathbb{R}^d,
  \quad a_0 = \max_{j \in J(i)} \vert a_{ij} \vert
\]
(see \eqref{eq:fi_def}). Hence, in particular, $g_i(x) > 0$ for any $x \in V \setminus \{ 0 \}$. Indeed,
if $g_i(x) \le 0$ for some $x \in V \setminus \{ 0 \}$, then taking into account the facts that the function $g_i$ is
positively homogeneous and $V$ is a cone one gets that $g_i(t x) = t g_i(x) \le 0$ and $tx \in V$ for any $t \ge 0$,
which contradicts the fact that $g_i$ is coercive on $V$.

Let us check that
\[
  \beta := \inf\big\{ g_i(x) \big\vert x \in V, \: \| x \| = 1 \big\} > 0.
\]
Suppose by contradiction that there exists a sequence $\{ x_n \} \subset V$ with $\| x_n \| = 1$ such that the sequence
$\beta_n := g_i(x_n)$ converges to zero. Note that $\beta_n > 0$, since $x_n \in V$. Define $y_n = (1/\beta_n) x_n$.
Clearly, $\{ y_n \} \subset V$, since $V$ is a cone, and $\| y_n \| \to + \infty$ as $n \to \infty$. Furthermore, taking
into account the fact that $g_i$ is a positively homogeneous function one gets that
\[
  g_i(y_n) = \frac{1}{\beta_n} g_i(x_n) = 1 \quad \forall n \in \mathbb{N},
\]
which contradicts the fact that $g_i$ is coercive on $V$. Therefore, $\beta > 0$ and $g_i(x) \ge \beta \| x \|$ for all
$x \in V$, since $V$ is a cone and the function $g_i$ is positively homogeneous.

Thus, the following lower estimate of the function $f_i = [f_i]_+$ holds true:
\[
  f_i(x) \ge g_i(x) - a_0 \ge \beta \| x \| - a_0 \quad \forall x \in V.
\] 
Fix any $x_0 \in S(f)$. Then for any $x \in V \setminus U(0, R)$ with $R = \max\{ \| x_0 \|, 2 a_0 / \beta \}$ one has
\begin{align*}
  f_i(x) \ge \beta \| x \| - a_0 &= \frac{\beta}{2} \| x \| + \frac{\beta}{2} \| x \| - a_0
  \ge \frac{\beta}{2} \| x \| 
  \\
  &\ge \frac{\beta}{4} \| x \| + \frac{\beta}{4} \| x_0 \|
  \ge \frac{\beta}{4} \| x  - x_0 \| \ge \frac{\beta}{4} \dist(x, S(f)).
\end{align*}
In turn, for any $x \in U(0, R)$ and for $\gamma = f_i^* / (R + \| x_0 \|)$ (recall that $i \in I \setminus I_0$, i.e.
$f_i^* > 0$) one has
\[
  \gamma \dist(x, S(f)) \le \gamma \| x - x_0 \| \le \gamma (R + \| x_0 \|) \le f_i^* \le f_i(x).
\]
Thus, inequality \eqref{eq:PositiveComponentsLowerEstimate} is satisfied with $\tau_i = \min\{ \gamma, \beta / 4 \}$.
\end{proof}

\begin{corollary}
Let $S(f)$ be nonempty and $f$ be coercive. Then $f$ has a global error bound.
\end{corollary}

\begin{corollary}
The function $f$ is coercive if and only if 
\begin{equation} \label{eq:StrongCoercivity}
  \liminf_{\| x \| \to + \infty} \frac{f(x)}{\| x \|} > 0
\end{equation}
\end{corollary}

\begin{proof}
Replacing, if necessary, $f$ with $f - C$ for a sufficiently large $C$ one can suppose that the set $S(f)$ is nonempty.
If $f$ is coercive, then by the previous corollary $f$ has a global error bound. Moreover, the coercivity of $f$ also
implies that the set $S(f)$ is bounded. Therefore, $0^+ S(f) = \{ 0 \}$ and applying 
Theorem~\ref{thrm:ErrorBound_GrowthAtInfinity} with $V = \mathbb{R}^d$ one can conclude that condition
\eqref{eq:StrongCoercivity} holds true, since by Theorem~\ref{thrm:ErrorBound_GrowthAtInfinity}
this condition is necessary for $f$ to have a global error bound, if $0^+ S(f) \cap \cl \cone V = \{ 0 \}$. In turn, if
condition \eqref{eq:StrongCoercivity} is satisfied, then $f$ is obviously coercive.
\end{proof}

The following example demonstrates that when $V$ is not a cone, the coercivity of $f$ on $V$ does not guarantee that $f$
has an error bound on $V$. Furthermore, it also shows that $f$ can have an error bound on an unbounded set $V$, but not
have and error bound on the conic hull of $V$, which means that, roughly speaking, necessary conditions for $f$ to have
an error bound on $V$ cannot be expressed in terms of the conic hull of $V$.

\begin{example}
Let $d = 2$, $f(x^{(1)}, x^{(2)}) = \min\{ \vert x^{(1)} \vert + \vert x^{(2)} \vert, 1 + \vert x^{(1)} \vert\}$, and
$\| \cdot \|$ be the Euclidean norm. Then $S(f) = \{ 0 \}$. Define $V = \{ x_n \}_{n \in \mathbb{N}}$, where 
$x_n = (n, n^2)$. Note that $f(x_n) = 1 + n$ for any $n \in \mathbb{N}$, that is, the function $f$ is coercive on $V$.
However, for any $\tau > 0$ one has
\[
  f(x_n) = 1 + n < \tau n^2 < \tau \| x_n \| = \tau \dist(x_n, S(f)) 
  \quad \forall n > \frac{1 + \sqrt{1 + 4 \tau}}{2\tau},
\] 
which implies that $f$ does not have an error bound on $V$.

Suppose now that $V = \{ (0, 1) \} \cup \{ (n, 0) \}_{n \in \mathbb{N}}$. As one can readily verify,
\[
  f(0, 1) = 1 = \dist\Big( (0, 1), S(f) \Big), \quad 
  f(n, 0) = n = \dist\Big( (n, 0), S(f) \Big)) \quad \forall n \in \mathbb{N},  
\]
that is, $f$ has an error bound on $V$. On the other hand, the ray $\{ (0, s) \mid s \ge 0 \}$ is obviously contained in
the conic hull of $V$ and for any $\tau > 0$ one has
\[
  f(0, s) = 1 < \tau s = \tau \dist\Big( (0, s), S(f) \Big) \quad \forall s > \max\left\{ \frac{1}{\tau}, 1 \right\}. 
\]
Thus, $f$ does not have an error bound on the conic hull of $V$.
\end{example}

In the case when $V$ is a convex cone and a representation of $f$ of the form \eqref{eq:MinMaxRepresentation} is known, 
one can provide geometric necessary and sufficient conditions for $f$ to have an error bound on $V$.

For any convex cone $K \subseteq \mathbb{R}^d$ denote by 
$K^* = \{ y \in \mathbb{R}^d \mid \langle y, x \rangle \le 0 \: \forall x \in K \}$ 
\textit{the polar cone} of $K$. Let also $\co Q$ be the convex hull of a set $Q \subset \mathbb{R}^d$.

\begin{theorem} 
Let $S(f)$ be nonempty and $V$ be a closed convex cone. Then for $f$ to have an error bound on $V$ it is sufficient that
for any $i \in I$ such that $f_i^* > 0$ one of the two following equivalent conditions holds true:
\begin{enumerate}
\item{$f_i$ is coercive on $V$; \label{st:Coercivity}}

\item{$0 \in \interior( \co\{ v_{ij} \mid j \in J(i) \} + V^* )$. \label{st:Geometric}}
\end{enumerate}
Moreover, these conditions become necessary, when $0^+ S(f) \cap V = \{ 0 \}$.
\end{theorem}

\begin{proof}
\textbf{Part 1.} Let us first verify that the two conditions from the formulation of the theorem are indeed equivalent.

\ref{st:Coercivity} ${}\implies{}$ \ref{st:Geometric}. Fix any $i \in I$ such that $f_i^* > 0$. Suppose by contradiction
that $0 \notin \interior( \co\{ v_{ij} \mid j \in J(i) \} + V^* )$. Then by the separation theorem there exists a
nonzero vector $z \in \mathbb{R}^d$ such that
\begin{equation} \label{eq:SeparationThrm}
  \langle z, v \rangle \le 0 \quad \forall v \in \co\{ v_{ij} \mid j \in J(i) \} + V^*.
\end{equation}
Note that $z \in V^{**} = V$. Indeed, if $z \notin V^{**}$, then one can find $y \in V^*$ such that 
$\langle z, y \rangle > 0$. On the other hand, taking into account \eqref{eq:SeparationThrm} and the fact that $V^*$ is
a cone one obtains that
\[
  \langle z, v_{ij} + t y \rangle \le 0 \quad \forall t > 0, \quad \forall j \in J(i),
\]
which obviously contradicts the inequality $\langle z, y \rangle > 0$.
 
From \eqref{eq:SeparationThrm} it follows that $\langle z, v_{ij} \rangle \le 0$ for all $j \in J(i)$. Thus, we have
found $z \in V$, $z \ne 0$, such that
\[
  f_i(t z) = \max_{j \in J(i)}(a_{ij} + t \langle v_{ij}, z \rangle) \le \max_{j \in J(i)} a_{ij} \quad \forall t \ge 0,
\] 
which contradicts the fact that $f_i$ is coercive on $V$.

\ref{st:Geometric} ${}\implies{}$ \ref{st:Coercivity}. Fix any $i \in I$ such that $f_i^* > 0$. By our assumption there
exists $r > 0$ such that
\[
  U(0, r) \subset \co\{ v_{ij} \mid j \in J(i) \} + V^*,
\]
which with the use of the definition of $V^*$ yields
\[
  r \| x \| \le \max\Big\{ \langle v, x \rangle \Bigm\vert v \in \co\{ v_{ij} \mid j \in J(i) \} + V^* \Big\} 
  \le \max_{j \in J(i)} \langle v_{ij}, x \rangle \quad \forall x \in V.  
\]
Consequently, for any $x \in V$ one has
\[
  f_i(x) = \max_{j \in J(i)} (a_{ij} + \langle v_{ij}, x \rangle) 
  \ge \max_{j \in J(i)} \langle v_{ij}, x \rangle - \max_{j \in J(i)} \vert a_{ij} \vert
  \ge r \| x \| - \max_{j \in J(i)} \vert a_{ij} \vert,  
\]
which obviously implies that $f_i$ is coercive on $V$.

\textbf{Part 2.} Let us now prove the main statement of the theorem. Suppose that each function $f_i$ with $f_i^* > 0$
is coercive on $V$. Then, as was shown in the proof of Theorem~\ref{thrm:ErrorBound_on_Cone}, for any such function
$f_i$ there exists $\tau_i > 0$ such that inequality \eqref{eq:PositiveComponentsLowerEstimate} holds true.

In turn, as was shown in the proof of Theorem~\ref{thrm:Robinson}, for any $i \in I_0$ the function $f_i$ has a global
error bound with some constant $\tau_i > 0$. Hence with the use of Lemma~\ref{lem:MinMax} one can easily check that $f$
has an error bound on $V$ with constant $\tau = \min_{i \in I} \tau_i$.

Suppose now that $f$ has an error bound on $V$ and $0^+ S(f) \cap V = \{ 0 \}$. Then by 
Theorem~\ref{thrm:ErrorBound_on_Cone} the function $f$ is coercive on $V$, which due to the equality 
$f = \min_{i \in I} f_i$ implies that each $f_i$ is coercive on $V$ as well.
\end{proof}

\subsection{Error bounds on unbounded sets: general conditions}
\label{subsect:EB_GeneralCond}

Let us finally provide general necessary and sufficient conditions for $f$ to have an error bound on an unbounded set
(in particular, a global error bound) that extend the results of Gowda \cite{Gowda1996} on global error bounds for
piecewise affine functions. 

One might be tempted to say that for $f$ to have a global error bound it is necessary that all ``flat pieces'' of $f$,
on which $f$ is positive (i.e. all polyhedral sets on which $f$ is constant and positive), are bounded. However, this is
not the case.

\begin{example} \label{ex:GlobalErrorBound_FlatPieces}
Let $d = 2$, the space $\mathbb{R}^2$ be endowed with the Euclidean norm, and
\[
  f(x) = \begin{cases}
    0, & \text{ if } x^{(1)} \le 0, 
    \\
    x^{(1)}, & \text{ if } 0 \le x^{(1)} \le 1, 
    \\
    1, & \text{ if } 1 \le x^{(1)} \le 2,
    \\
    x^{(1)} - 1, & \text{ if } x^{(1)} \ge 2.
  \end{cases}
\]
Then $S(f) = \{ x \in \mathbb{R}^2 \mid x^{(1)} \le 0 \}$ and $\dist(x, S(f)) = [x^{(1)}]_+$ for all 
$x \in \mathbb{R}^2$. Hence, as is easily seen,
\[
  \frac{1}{2} \dist(x, S(f)) \le f(x) \quad \forall x \in \mathbb{R}^2,
\]
that is, $f$ has a global error bound, despite the fact that $f$ is constant and positive on the unbounded set 
$\{ x \in \mathbb{R}^2 \mid 1 \le x^{(1)} \le 2 \}$.
\end{example}

Our aim is to show that the existence of a global error bound for $f$ is completely defined by the location of
directions along which $f$ is constant and positive with respect to the set $S(f)$. Namely, $f$ has a global error bound
if and only if such directions are, in a sense, parallel to the set $S(f)$. To conveniently formulate this condition, we
will use the recession function of the function $[f]_+$, being inspired by the approach of Gowda \cite{Gowda1996}.

Let us formulate necessary and sufficient conditions for $f$ to have an error bound on a finite union of polyhedral sets
(e.g. to have a global error bound).

\begin{theorem} \label{thrm:ErrorBound_RecessionCones}
Let $S(f)$ be nonempty and $V \subseteq \mathbb{R}^d$ be a finite union of polyhedral sets. Then for $f$ to have an
error bound on $V$ it is necessary and sufficient that 
\[
  S([f]_+^{\infty}) \cap 0^+ V \subseteq 0^+ S(f).
\]
\end{theorem}

We divide the proof of this theorem into a series of lemmas. First, we reformulate the statement of the theorem in terms
of the convex piecewise affine functions $f_i$, $i \in I$.

\begin{lemma} \label{lem:PositiveErrorBound}
Let the assumptions of Theorem~\ref{thrm:ErrorBound_RecessionCones} hold true. Then $f$ has an error bound on $V$ if and
only if for any $i \in I$ with $f_i^* > 0$ there exists $\tau_i > 0$ such that
\begin{equation} \label{eq:StratifiedErrorBound}
  f_i(x) \ge \tau_i \dist(x, S(f)) \quad \forall x \in V.
\end{equation}
\end{lemma}

\begin{proof}
By equality \eqref{eq:MinMaxRepresentation} and the definition of error bound, the function $f$ has an error bound on
$V$ if and only if there exists $\tau > 0$ such that
\[
  f(x) = \min_{i \in I} f_i(x) \ge \tau \dist(x, S(f)) \quad \forall x \in V.
\]
Therefore, $f$ has an error bound on $V$ if and only if for each $i \in I$ one can find $\tau_i > 0$ for which
inequality \eqref{eq:StratifiedErrorBound} holds true. It remains to note that, as was shown in the proof of
Theorem~\ref{thrm:Robinson}, for any $i \in I$ with $f_i^* \le 0$ the function $f_i$ has a global error bound, which
with the use of Lemma~\ref{lem:MinMax} implies that inequality \eqref{eq:StratifiedErrorBound} is satisfied for any 
$i \in I_0$ with $\tau_i = \tau(f_i, \mathbb{R}^d)$.
\end{proof}

The second step is to reformulate condition \eqref{eq:StratifiedErrorBound} in geometric terms involving recession cones
of some sets. To this end, we need to prove three auxiliary results on the distance to a finite union of polyhedral
sets. 

\begin{lemma} \label{lem:BoundedDistFunc}
Let $X \subseteq \mathbb{R}^d$ be a finite union of polyhedral sets and $x, z \in \mathbb{R}^d$ be fixed. Then 
the function $\lambda \mapsto \dist(x + \lambda z, X)$ is bounded on $[0, + \infty)$, if $z \in 0^+ X$, and
$\dist(x + \lambda z, X) \to + \infty$ as $\lambda \to + \infty$, if $z \notin 0^+ X$.
\end{lemma}

\begin{proof}
Fix any $x, z \in \mathbb{R}^d$. If $z \in 0^+ X$, then there exists $x_0 \in X$ such that $x_0 + \lambda z \in X$ for
all $\lambda \ge 0$. Therefore
\[
  \dist(x + \lambda z, X) \le \| x + \lambda z - (x_0 + \lambda z) \| = \| x - x_0 \| \quad \forall \lambda \ge 0,
\]
i.e. the function $\lambda \mapsto \dist(x + \lambda z, X)$ is bounded on $[0, + \infty)$.

Suppose now that $z \notin 0^+ X$. By our assumption $X$ is the union of some polyhedral sets 
$X_1, \ldots, X_s \subset \mathbb{R}^d$. As is easily seen,
\[
  \dist(y, X) = \min\Big\{ \dist(y, X_1), \ldots, \dist(y, X_s) \Big\} \quad \forall y \in \mathbb{R}^d.
\]
Therefore, if we prove that $\dist(x + \lambda z, X_i) \to + \infty$ as $\lambda \to + \infty$ for any 
$i \in \{ 1, \ldots, s \}$, we can conclude that $\dist(x + \lambda z, X) \to + \infty$ as $\lambda \to + \infty$ as 
well.

Fix any $i \in \{ 1, \ldots, s \}$. Clearly, $0^+ X_i \subseteq 0^+ X$, which implies that $z \notin 0^+ X_i$. By
the Motzkin theorem one has $X_i = P_i + 0^+ X_i$ for some polytope $P_i \subset \mathbb{R}^d$, since $X_i$ is a
polyhedral set. Therefore, by the reverse triangle inequality for any $\lambda \ge 0$ one has
\begin{align*}
  \dist(x + \lambda z, X_i) 
  &= \inf\Big\{ \| x + \lambda z - (y_1 + y_2) \| \Bigm\vert y_1 \in P_i, y_2 \in 0^+ X_i \Big\}
  \\ 
  &\ge \inf\Big\{ \| \lambda z - y_2 \| - \| x \| - \| y_1 \| \Bigm\vert y_1 \in P_i, y_2 \in 0^+ X_i \Big\}
  \\
  &\ge \dist(\lambda z, 0^+ X_i) - \| x \| - \max_{y \in P_i} \| y \|.
\end{align*}
Since $0^+ X_i$ is a cone, the function $\dist(\cdot, 0^+ X_i)$ is obviously positively homogeneous. Hence
\[
  \dist(x + \lambda z, X_i) \ge \lambda r - \| x \| - \max_{y \in K_i} \| y \| \quad \forall \lambda \ge 0,
\]
where $r = \dist(z, 0^+ X_i) > 0$ (recall that $z \notin 0^+ X_i$). Thus, $\dist(x + \lambda z, X_i) \to + \infty$, and
the proof is complete.
\end{proof}

\begin{lemma} \label{lem:Dist_PolyhedralCones}
Let $Y, V \subseteq \mathbb{R}^d$ be polyhedral convex cones. Then there exists $C > 0$ such that
\begin{equation} \label{eq:EquivalentDist}
  C \dist(x, Y) \ge \dist(x, Y \cap V) \quad \forall x \in V.
\end{equation}
\end{lemma}

\begin{proof}
We prove the lemma in the case when $\| \cdot \|$ is the Euclidean norm. Clearly, the validity of the lemma in the
general case follows directly from its validity in the Euclidean case.

Choose some $C > 0$ and introduce the functions
\[
  p_C(x) = C \dist(x, Y), \quad q(x) = \dist(x, Y \cap V) \quad \forall x \in \mathbb{R}^d.
\]
Since $Y$ and $V$ are convex cones, the functions $p_C$ and $q$ are sublinear. Recall that a sublinear function is
equal to the support function of its subdifferential at the origin (see, e.g. \cite[Thm.~V.3.1.1]{Lemarechal}), that is
\[
  p_C(x) = \max_{v \in \partial p_C(0)} \langle v, x \rangle, \quad
  q(x) = \max_{w \in \partial q(0)} \langle w, x \rangle \quad \forall x \in \mathbb{R}^d.
\]
Therefore, inequality \eqref{eq:EquivalentDist} holds true if and only if
\[
  \max_{v \in \partial p_C(0) - w} \langle v, x \rangle \ge 0 \quad \forall x \in V
  \quad \forall w \in \partial q(0).
\]
With the use of the separation theorem one can readily check that this inequality is satisfied if and only if the set
$(\partial p_C(0) - w) \cap (- V^*)$ is nonempty for any $w \in \partial q(0)$ or, equivalently, if and only if
\[
  0 \in \partial p_C(0) - w + V^* \quad \forall w \in \partial q(0).
\]
The inclusion above can be rewritten as
\begin{equation} \label{eq:DistQuasidiffCond}
  \partial q(0) \subseteq \partial p_C(0) + V^*.
\end{equation}
Thus, inequality \eqref{eq:EquivalentDist} is satisfied for some $C > 0$ if and only if inclusion
\eqref{eq:DistQuasidiffCond} is satisfied for the same $C > 0$. Let us prove that this inclusion holds true, provided 
$C > 0$ is large enough.

Indeed, note that by \cite[Example~2.130]{BonnansShapiro} one has
\[
  \partial p_C(0) = \Big\{ v \in Y^* \Bigm\vert \| v \| \le C \Big\}, \quad
  \partial q(0) = \Big\{ w \in (Y \cap V)^* \Bigm\vert \| w \| \le 1 \Big\},
\]
Moreover, by \cite[Cor.~16.4.2]{Rockafellar} (see also \cite[Thm.~16.4 and Thm.~20.1]{Rockafellar}) one has
$(Y \cap V)^* = Y^* + V^*$, since both $Y$ and $V$ are polyhedral cones. Thus, inclusion
\eqref{eq:DistQuasidiffCond} can be rewritten as
\[
  \Big\{ w \in Y^* + V^* \Bigm\vert \| w \| \le 1 \Big\} \subseteq 
  \Big\{ v \in Y^* \Bigm\vert \| v \| \le C \Big\} + V^*.
\]
Consequently, it is sufficient to prove that for any $w \in Y^* + V^*$ with $\| w \| \le 1$ one can find $v_1 \in Y^*$
with $\| v_1 \| \le C$ and $v_2 \in V^*$ such that $w = v_1 + v_2$. Note that if $Y^* + V^* = \{ 0 \}$, then this claim
is obvious. Therefore, one can suppose that $Y^* + V^* \ne \{ 0 \}$.

To prove the existence of the required $C > 0$, note that the cones $Y^*$ and $V^*$ are polyhedral by
\cite[Cor.~19.2.2]{Rockafellar}, since the cones $Y$ and $V$ are polyhedral. Therefore by \cite[Thm.~19.1]{Rockafellar}
both cones $Y^*$ and $V^*$ are finitely generated, that is, they are the convex conic hulls of some vectors 
$y_1, \ldots, y_n \in Y^*$ and $z_1, \ldots, z_k \in V^*$. 

Let $\mathcal{M}$ be the collection of all nonempty subsets $M$ of the set 
$\{ 1, \ldots, n \} \times \{ 1, \ldots, k \}$ such that the vectors $y_i + z_j$, $(i, j) \in M$, are linearly
independent. Note that the set $\mathcal{M}$ is nonempty, since, as was noted above, one can assume that 
$Y^* + V^* \ne \{ 0 \}$ and the cone $Y^* + V^*$ is obviously the convex conic hull of
the vectors $y_i + z_j$, $i \in \{ 1, \ldots, n \}$, $j \in \{ 1, \ldots, k \}$.

For any $M \in \mathcal{M}$ introduce linear subspace $\mathcal{E}_M = \linhull\{ y_i + z_j \mid (i, j) \in M \}$. By
definition, for any $x \in \mathcal{E}_M$ there exist unique $\alpha(i, j) \in \mathbb{R}$, $(i, j) \in M$, such that
$x = \sum_{(i, j) \in M} \alpha_{ij} (y_i + z_j)$. Denote by $\| x \|_M = \sum_{(i, j) \in M} \vert \alpha_{ij} \vert$.
Clearly, $\| \cdot \|_M$ is a norm on $\mathcal{E}_M$. Therefore, it is equivalent to the Euclidean norm, which, in
particular, implies that there exists $C_M > 0$ such that $\| x \|_M \le C_M \| x \|$ for all $x \in \mathcal{E}_M$.

Now, fix any $w \in Y^* + V^*$ with $\| w \| \le 1$. Since the cone $Y^* + V^*$ is the convex conic hull of
the vectors $\{ y_i + z_j \}$, by the version of Carath\'{e}odory's theorem for convex cones 
\cite[Cor.~17.1.2]{Rockafellar} there exists $M \in \mathcal{M}$ such that $w$ can be represented as the convex conic
combination of the vectors $\{ y_i + z_j \}$, $(i, j) \in M$, that is, one can find $\alpha_{ij} \ge 0$, $(i, j) \in M$,
such that  $w = \sum_{(i, j) \in M} \alpha_{ij} (y_i + z_j)$. Moreover, one has 
$\sum_{(i, j) \in M} \alpha_{ij} \le C_M$. 

Define 
\[
  v_1 = \sum_{(i, j) \in M} \alpha_{ij} y_i, \quad
  v_2 = \sum_{(i, j) \in M} \alpha_{ij} z_j. 
\]
Then $v_1 \in Y^*$, $v_2 \in V^*$, and $w = v_1 + v_2$. Moreover, one has
\[
  \| v_1 \| \le \sum_{(i, j) \in M} \alpha_{ij} \| y_i \| \le C_M \max\Big\{ \| y_1 \|, \ldots, \| y_n \| \Big\}. 
\]
Thus, we have proved that for any $w \in Y^* + V^*$ with $\| w \| \le 1$ one can find $v_1 \in Y^*$ and $v_2 \in V^*$
such that $w = v_1 + v_2$ and $\| v_1 \| \le C$, where 
\[
  C = \Big( \max_{M \in \mathcal{M}} C_M \Big) \max\Big\{ \| y_1 \|, \ldots, \| y_n \| \Big\}.
\]
Note that $C < + \infty$, since the collection $\mathcal{M}$ consists of a finite number of sets.
\end{proof}

As the following example shows, the conclusion of the previous lemma does not hold true in the case when either of the
cones $Y$ and $V$ is not polyhedral.

\begin{example}
Let $d = 3$, $Y = \{ x \in \mathbb{R}^3 \mid x^{(1)} = 0 \}$, and $V$ be the convex conic hull of the disc
\[
  D = \Big\{ x \in \mathbb{R}^3 \Bigm\vert (x^{(1)} - 1)^2 + (x^{(2)})^2 \le 1, \: x^{(3)} = 1 \Big\}.
\]
Note that the cone $V$ is not polyhedral in this case.

Let $\| \cdot \|$ be the Euclidean norm. Then $Y \cap V = \{ x \in \mathbb{R}^3 \mid x^{(1)} = x^{(2)} = 0 \}$ and, as
is easily seen, 
\[
  \dist(x, Y \cap V) = \sqrt{(x^{(1)})^2 + (x^{(2)})^2}, \quad
  \dist(x, Y) = \vert x^{(1)} \vert \quad \forall x \in V.
\]
Observe that for $x(t) = (1 + \sin t, \cos t, 1) \in V$ the inequality 
\[
  \dist(x(t), Y \cap V) \le C \dist(x(t), Y) \quad \forall t \in \mathbb{R}
\]
is not satisfied for any $C > 0$, since
\begin{multline*}
  \dist(x(t), Y \cap V)^2 = 2(1 + \sin t) > C^2 (1 + \sin t)^2 = C^2 \dist(x(t), Y)^2
  \\
  \forall t \in \left( \frac{3 \pi}{2} - \varepsilon, \frac{3 \pi}{2} \right) 
  \cup \left( \frac{3 \pi}{2}, \frac{3 \pi}{2} + \varepsilon \right)
\end{multline*}
for any sufficiently small $\varepsilon > 0$, depending on $C$, due to the fact that
\[
  \lim_{t \to \frac{3 \pi}{2}} \frac{2(1 + \sin t)}{(1 + \sin t)^2} = + \infty.
\] 
Thus, the conclusion of
Lemma~\ref{lem:Dist_PolyhedralCones} does not hold true when the cone $V$ is not polyhedral. Switching $Y$ and $V$ one
can check that the conclusion of this lemma does not hold true in the case when the cone $Y$ is not polyhedral either.
\end{example}

Now we are ready to prove the main result on the distance to a finite union of polyhedral sets that is the key part of
the proof of Theorem~\ref{thrm:ErrorBound_RecessionCones}.

\begin{lemma} \label{lem:Dist_RecessionCones}
Let $X, Y, V \subseteq \mathbb{R}^d$ be finite unions of polyhedral sets. Then the inequality
\begin{equation} \label{eq:DistVsDist}
  \dist(x, X) \le C \dist(x, Y) + \theta \quad \forall x \in V
\end{equation}
is satisfied for some $C > 0$ and $\theta \ge 0$ if and only if $0^+(Y \cap V) \subseteq 0^+ X$.
\end{lemma}

\begin{proof}
Let inequality \eqref{eq:DistVsDist} hold true for some $C > 0$ and $\theta \ge 0$. Suppose by contradiction that there
exists $z \in 0^+(Y \cap V)$ such that $z \notin 0^+ X$. Then by the definition of the recession cone
$x_0 + \lambda z \in Y \cap V$ for some $x_0 \in Y \cap V$ and all $\lambda \ge 0$, which yields 
$\dist(x_0 + \lambda z, Y) = 0$ for all $\lambda \ge 0$. On the other hand, by Lemma~\ref{lem:BoundedDistFunc} one has
$\dist(x_0 + \lambda z, X) \to + \infty$ as $\lambda \to + \infty$, which obviously contradicts inequality
\eqref{eq:DistVsDist}.

Suppose now that $0^+(Y \cap V) \subseteq 0^+ X$. We will prove inequality \eqref{eq:DistVsDist} by reducing the proof
to the particular case when $Y$ and $V$ are polyhedral cones.

Indeed, by our assumption 
\begin{equation} \label{eq:PolyhedralUnion}
  X = \bigcup_{i = 1}^{n} X_i, \quad Y = \bigcup_{j = 1}^s Y_j, \quad V = \bigcup_{k = 1}^t V_k
\end{equation}
for some $n, s, t \in \mathbb{N}$ and some polyhedral sets $X_i, Y_j, V_k \subseteq \mathbb{R}^d$. As is easily seen,
one has
\[
  0^+ X = \bigcup_{i = 1}^{n} 0^+ X_i, \quad 
  0^+ (Y \cap V) = \bigcup_{j = 1}^s \bigcup_{k = 1}^t 0^+ (Y_j \cap V_k)
\]
and, furthermore, $0^+ (Y_j \cap V_k) = 0^+ Y_j \cap 0^+ V_k$ (see, e.g. \cite[Cor.~8.3.3]{Rockafellar}). Thus, for
any $j \in J_Y := \{ 1, \ldots, s \}$ and $k \in K_V := \{ 1, \ldots, t \}$ one has
\[
  0^+ Y_j \cap 0^+ V_k \subseteq 0^+ X,
\]
which obviously implies that
\begin{equation} \label{eq:DistIneq_ReccCones_Inter}
  \dist(x, 0^+ Y_j \cap 0^+ V_k) \ge \dist(x, 0^+ X) \quad \forall x \in \mathbb{R}^d.
\end{equation}
Note that the recession cones $0^+ Y_j$ and $0^+ V_k$ are polyhedral by \cite[Thm.~19.5]{Rockafellar}, since the sets
$Y_j$ and $V_k$ are polyhedral. Therefore, by Lemma~\ref{lem:Dist_PolyhedralCones} for any $j \in J_Y$ and $k \in K_V$
there exists $C_{jk} > 0$ such that
\[
  C_{jk} \dist(x, 0^+ Y_j) \ge \dist(x, 0^+ Y_j \cap 0^+ V_k) \quad \forall x \in 0^+ V_k. 
\]
Hence with the use of \eqref{eq:DistIneq_ReccCones_Inter} one obtains that for any $j \in J_Y$ the following inequality
holds true:
\[
  \Big( \max_{k \in K_V} C_{jk} \Big) \dist(x, 0^+ Y_j) \ge \dist(x, 0^+ X)
  \quad \forall x \in 0^+ V = \bigcup_{k = 1}^t 0^+ V_k.
\]
Finally, taking into account the obvious equality
\[
  \dist(x, 0^+ Y) = \dist\Big(x, \bigcup_{j = 1}^s 0^+ Y_j \Big) = \min_{j \in J_Y} \dist(x, 0^+ Y_j)
  \quad \forall x \in \mathbb{R}^d
\]
one gets that
\begin{equation} \label{eq:DistIneq_ReccCones}
  C \dist(x, 0^+ Y) \ge \dist(x, 0^+ X) \quad \forall x \in 0^+ V,
\end{equation}
where $C = \max\{ C_{jk} \mid j \in J_Y, \: k \in K_V \}$.

By the Motzkin theorem for each $i \in N_X := \{ 1, \ldots, n \}$, $j \in J_Y$, and $k \in K_V$ there exist polytopes
$P(X_i), P(Y_j), P(V_k) \subset \mathbb{R}^d$ such that
\[
  X_i = P(X_i) + 0^+ X_i, \quad Y_j = P(Y_j) + 0^+ Y_j, \quad V_k = P(V_k) + 0^+ V_k,
\]
since the sets $X_i$, $Y_j$, and $V_k$ are polyhedral by our assumption. Define
\[
  \theta_1 = \max\Big\{ \max_{i \in N_X} \max_{v \in P(X_i)} \| v \|, 
  \: \max_{j \in J_Y} \max_{w \in P(Y_j)} \| w_j \| \Big\}.
\]
As is easily seen, by the reverse triangle inequality for any $x \in \mathbb{R}^d$ one has
\begin{equation} \label{eq:DistVsDistToReccCone}
  \big\vert \dist(x, X_i) - \dist(x, 0^+ X_i) \big\vert \le \theta_1, \quad
  \big\vert \dist(x, Y_j) - \dist(x, 0^+ Y_j) \big\vert \le \theta_1 
\end{equation}
for any $i \in N_X$ and $j \in J_Y$. Put also $\theta_2 = \max\{ \| x \| \mid x \in P(V_k), k \in K_V \}$ and
\[
  \theta_3 = \max\Big\{ \max_{i \in N_X, k \in K_V} \max_{x \in P(V_k)} \dist(x, 0^+ X_i), \: 
  \max_{j \in J_Y, k \in K_V} \max_{x \in P(V_k)} \dist(x, 0^+ Y_j) \Big\}.
\]
Clearly, $\theta_3 < + \infty$, since the sets $P(V_k)$ are compact and the corresponding distance functions are
continuous.

Choose any $k \in K_V$ and $x \in V_k$. Then $x = x_1 + x_2$ for some $x_1 \in P(V_k)$ and $x_2 \in 0^+ V_k$.
Applying inequalities \eqref{eq:DistIneq_ReccCones} and \eqref{eq:DistVsDistToReccCone}, and the fact that the distance
to a convex cone is a sublinear function one gets that
\begin{align*}
  \dist&(x, X) = \min_{i \in N_X} \dist(x, X_i)
  \le \theta_1 + \min_{i \in N_X} \dist(x, 0^+ X_i) 
  \\
  &\le \theta_1 + \theta_3 + \min_{i \in N_X} \dist(x_2, 0^+ X_i)
  = \theta_1 + \theta_3 + \dist(x_2, 0^+ X)
  \\
  &\le \theta_1 + \theta_3 + C \dist(x_2, 0^+ Y) = \theta_1 + \theta_3 + C \min_{j \in J_Y} \dist(x_2, 0^+ Y_j)
  \\
  &\le \theta_1 + C \theta_2 + \theta_3 + C \min_{j \in J_Y} \dist(x, 0^+ Y_j)
  \\
  &\le 2\theta_1 + C \theta_2 + \theta_3 + C \min_{j \in J_Y} \dist(x, Y_j)
  = 2 \theta_1 + C \theta_2 + \theta_3 + C \dist(x, Y).
\end{align*}
Since $k \in K_V$ and $x \in V_k$ were chosen arbitrary, one can conclude that inequality \eqref{eq:DistVsDist} holds
true with $\theta = 2 \theta_1 + C \theta_2 + \theta_3$. 
\end{proof}

Lemmas~\ref{lem:BoundedDistFunc} and \ref{lem:Dist_RecessionCones} enable us to reformulate condition
\eqref{eq:StratifiedErrorBound} from Lemma~\ref{lem:PositiveErrorBound} in geometric terms.

\begin{lemma} \label{lem:ErrorBound_StratifiedReccCones}
Under the assumptions of Lemma~\ref{lem:PositiveErrorBound} the function $f$ has an error bound on $V$ if and only if
for any $i \in I$ such that $f_i^* > 0$ one has 
\begin{equation} \label{eq:PosConvex_V_Inclusion}
  0^+ (S(f_i - f_i^*) \cap V) \subseteq 0^+ S(f).
\end{equation}
\end{lemma}

\begin{proof}
Fix any $i \in I$ such that $f_i^* > 0$. Our aim is to show that the inequality 
\begin{equation} \label{eq:PosConvexErrorBound}
  f_i(x) \ge \tau_i \dist(x, S(f)) \quad \forall x \in V
\end{equation}
is satisfied if and only if the inclusion \eqref{eq:PosConvex_V_Inclusion} hold true. Then applying 
Lemma~\ref{lem:PositiveErrorBound} one obtains the required result.

Let inequality \eqref{eq:PosConvexErrorBound} be satisfied for some $\tau_i > 0$. Suppose by contradiction that
there exists $z \in 0^+ (S(f_i - f_i^*) \cap V)$ such that $z \notin 0^+ S(f)$. By the definition
of the recession cone one can find $x_0 \in S(f_i - f_i^*) \cap V$ such that $x_0 + \lambda z \in S(f_i - f_i^*) \cap V$
for all $\lambda \ge 0$. Hence, in particular, $f_i(x_0 + \lambda z) = f_i^*$ for all $\lambda \ge 0$. In turn, from
Lemma~\ref{lem:BoundedDistFunc} and the condition  $z \notin 0^+ S(f)$ it follows that 
$\dist(x_0 + \lambda z, S(f)) \to + \infty$ as $\lambda \to + \infty$, which contradicts the fact that 
by our assumtion
\[
  f_i^* = f_i(x_0 + \lambda z) \ge \tau_i \dist(x_0 + \lambda z, S(f)) \quad \forall \lambda \ge 0.
\]
Thus, condition \eqref{eq:PosConvex_V_Inclusion} holds true.

Conversely, suppose that the inclusion \eqref{eq:PosConvex_V_Inclusion} holds true. Then by
Lemma~\ref{lem:Dist_RecessionCones} there exist $C > 0$ and $\theta > 0$ such that
\begin{equation} \label{eq:DistUpperEstimate}
  \dist(x, S(f)) \le \theta + C \dist(x, S(f_i - f_i^*)) \quad \forall x \in V.
\end{equation}
Applying Hoffman's Theorem in precisely the same way as in the proof of Theorem~\ref{thrm:Robinson} one gets that 
the function $f_i - f_i^*$ has a global error bound with some constant $\tau_i > 0$, that is,
\[
  f_i(x) \ge f_i^* + \tau_i \dist( x, S(f_i - f_i^*) ) \quad \forall x \in \mathbb{R}^d.
\]
Recall that $f_i^* > 0$ by our assumption. Consequently, decreasing $\tau_i > 0$, if necessary, one can suppose that
$f_i^* / \tau_i \ge \theta / C$, which with the use of \eqref{eq:DistUpperEstimate} implies that
\[
  f_i(x) \ge \tau_i \Big( \frac{f_i^*}{\tau_i} + \dist( x, S(f_i - f_i^*) ) \Big)
  \ge \frac{\tau_i}{C} \dist(x, S(f)) \quad \forall x \in V.
\]
Thus, the proof is complete.
\end{proof}

Finally, it remains to combine all the lemmas above into a coherent proof.

\textbf{Proof of Theorem~\ref{thrm:ErrorBound_RecessionCones}}: If the function $f$ has an error bound on $V$, then by
Lemma~\ref{lem:ErrorBound_StratifiedReccCones} the inclusion $0^+ (S(f_i - f_i^*) \cap V) \subseteq 0^+ S(f)$ holds true
for any $i \in I$ such that $f_i^* > 0$. Hence with the use of Lemma~\ref{lem:RecFuncSublevel} and the fact that
\begin{equation} \label{eq:RecConeIntersection}
  0^+(X \cap V) = 0^+ X \cap 0^+ V
\end{equation}
for any set $X \subset \mathbb{R}^d$ that is a finite union of closed convex sets 
(see, e.g. \cite[Cor.~8.3.3]{Rockafellar}) one obtains that
\begin{align*}
  S([f]_+^{\infty}) \cap 0^+ V &= \Big( 0^+ S(f) \cap 0^+ V \Big) 
  \cup \Big( \bigcup_{i \in I \colon f_i^* > 0} 0^+ S(f_i - f_i^*) \cap 0^+ V \Big)
  \\
  &\subseteq 0^+ S(f) \cup \Big( \bigcup_{i \in I \colon f_i^* > 0} 0^+ \big( S(f_i - f_i^*) \cap V \big) \Big)
  \subseteq 0^+ S(f).
\end{align*}
Let us prove the converse statement. Suppose that $S([f]_+^{\infty}) \cap 0^+ V \subseteq 0^+ S(f)$. Then by equality
\eqref{eq:RecConeIntersection} and Lemma~\ref{lem:RecFuncSublevel} for any $i \in I$ such that $f_i^* > 0$ one has
\[
  0^+ (S(f_i - f_i^*) \cap V) = 0^+ S(f_i - f_i^*) \cap 0^+ V \subseteq S([f]_+^{\infty}) \cap 0^+ V
  \subseteq 0^+ S(f),
\]
which by Lemma~\ref{lem:ErrorBound_StratifiedReccCones} implies that $f$ has an error bound on $V$. \qed

Let us illustrate Theorem~\ref{thrm:ErrorBound_RecessionCones} and some of the lemmas above by applying them to 
two simple examples.

\begin{example} \label{ex:GlobalErrorBound_RecCones}
Let $d = 2$ and $f$ be defined as in Example~\ref{ex:GlobalErrorBound_FlatPieces}, that is,
\[
  f(x) = \begin{cases}
    0, & \text{ if } x^{(1)} \le 0, 
    \\
    x^{(1)}, & \text{ if } 0 \le x^{(1)} \le 1, 
    \\
    1, & \text{ if } 1 \le x^{(1)} \le 2,
    \\
    x^{(1)} - 1, & \text{ if } x^{(1)} \ge 2.
  \end{cases}
\]
Then 
\[
  [f]_+^{\infty}(x) = \begin{cases}
    0, & \text{ if } x^{(1)} \le 0,
    \\
    x^{(1)}, & \text{ if } x^{(1)} > 0
  \end{cases}
  = \max\{ 0, x^{(1)} \}.
\]
Therefore $S(f) = S([f]_+^{\infty}) = \{ x \in \mathbb{R}^2 \mid x^{(1)} \le 0 \}$ and $f$ has a global error bound 
by Theorem~\ref{thrm:ErrorBound_RecessionCones}.

Let us now apply Lemma~\ref{lem:ErrorBound_StratifiedReccCones}. Observe that $f(x) = \min\{ f_1(x), f_2(x) \}$ with
\[
  f_1(x) = \max\{ 0, x^{(1)} \}, \quad f_2(x) = \max\{ 1, x^{(1)} - 1 \}. 
\]
Clearly, $f_1^* = 0$ and $f_2^* = 1$. Moreover, one has 
\[
  S(f_2 - f_2^*) = \{ x \in \mathbb{R}^2 \mid x^{(1)} \le 2 \}. 
\]
Therefore $0^+ S(f_2 - f_2^*) = 0^+ S(f)$ and one can conclude that $f$ has a global error bound by
Lemma~\ref{lem:ErrorBound_StratifiedReccCones}.
\end{example}

\begin{example}
Let $d = 2$ and $f(x) = \min\{ f_1(x), f_2(x) \}$, where
\[
  f_1(x) = \max\{ 0, x^{(1)} \}, \quad f_2(x) = \max\{ 1 + x^{(2)}, 1 - x^{(2)} \}.
\]
Let us apply Lemma~\ref{lem:ErrorBound_StratifiedReccCones} first. Indeed, one has $f_1^* = 0$, $f_2^* = 1$, and
\begin{gather*}
  S(f) = 0^+ S(f) = \{ x \in \mathbb{R}^2 \mid x^{(1)} \le 0 \}, \quad
  \\
  S(f_2 - f_2^*) = 0^+ S(f_2 - f_2^*) = \{ x \in \mathbb{R}^2 \mid x^{(2)} = 0 \}.
\end{gather*}
Consequently, by Lemma~\ref{lem:ErrorBound_StratifiedReccCones} the function $f$ does not have a global error bound, but
has an error bound on any finite union of polyhedral sets $V \subset \mathbb{R}^2$ having bounded intersection with the
ray $K := \{ x \in \mathbb{R}^2 \mid x^{(1)} \ge 0, \: x^{(2)} = 0 \}$.

Let us now apply Theorem~\ref{thrm:ErrorBound_RecessionCones}. As is easy to check, one has
\[
  [f]_+^{\infty} = \min\{ [x^{(1)}]_+, \vert x^{(2)} \vert \}.
\]
(see \eqref{eq:RecFuncViaMinMaxRepresentation}). Therefore
\[
  S([f]_+^{\infty})
  = \Big\{ x \in \mathbb{R}^2 \Bigm\vert x^{(1)} \le 0 \text{ or } x^{(2)} = 0 \} \Big\}.
\]
Note that $S([f]_+^{\infty})$ is not contained in $0^+ S(f)$. Consequently, by 
Theorem~\ref{thrm:ErrorBound_RecessionCones} the function $f$ does not have a global error bound, but has an error one
any finite union of polyhedral sets $V \subset \mathbb{R}^2$ having bounded intersection with the ray $K$.
\end{example}

\section{Error bounds for systems of piecewise affine equalities and inequalities}
\label{sect:PAConstraints}

To conveniently summarize the main results of this article, let us apply them to obtain a straightforward extension of
Hoffman's theorem \cite{Hoffman} to the case of systems of piecewise affine equalities and inequalities.

Let $F \colon \mathbb{R}^d \to \mathbb{R}^m$ and $G \colon \mathbb{R}^d \to \mathbb{R}^n$ be piecewise affine
functions, and denote by
\[
  \Omega = \{ x \in \mathbb{R}^d \mid F(x) = 0, \: G(x) \le 0 \}
\]
the solution set of the system $F(x) = 0$, $G(x) \le 0$. Here the inequality is understood coordinate-wise.

\begin{theorem}
Let $\Omega$ be nonempty, $V \subset \mathbb{R}^d$ be a given set, and $\| \cdot \|$ be an arbitrary norm on
$\mathbb{R}^m$. Then the inequality
\begin{equation} \label{eq:ErrorBound_PAffSystem}
  \tau \dist(x, \Omega) \le \varphi(x) := \| F(x) \| + \sum_{i = 1}^n \max\{ G_i(x), 0 \} 
  \quad \forall x \in V
\end{equation}
is satisfied for some $\tau > 0$, provided one of the following five conditions holds true:
\begin{enumerate}
\item{$V$ is bounded;}

\item{$V$ is unbounded and 
\[
  \liminf_{\| x \| \to + \infty, x \in V} \frac{\varphi(x)}{\| x \|} > 0;
\]
}

\item{$V$ is a cone and $\varphi$ is coercive on $V$;}

\item{$V = \mathbb{R}^d$ and both $F$ and $G$ are positively homogeneous;}

\item{$V$ is a finite union of polyhedral sets and $S(\varphi^{\infty}) \cap 0^+ V \subseteq 0^+ \Omega$.}
\end{enumerate}
Furthermore, the last condition is necessary for inequality \eqref{eq:ErrorBound_PAffSystem} to hold true in the general
case, while the second and third conditions are necessary for this inequality to hold true in the case 
$0^+ \Omega \cap \cl \cone V = \{ 0 \}$.
\end{theorem}

\begin{proof}
Let $\| \cdot \|_{\infty}$ be the $\ell_{\infty}$ norm on $\mathbb{R}^d$. Introduce the function
\[
  f(x) = \| F(x) \|_{\infty} + \sum_{i = 1}^n \max\{ G_i(x), 0 \} \quad \forall x \in \mathbb{R}^d.
\]
This function is obviously piecewise affine, $S(f) = \Omega$, and $[f]_+ = f$. Note also that the function
\[
  \varphi^{\infty}(x) = \lim_{\lambda \to + \infty} \frac{\varphi(\lambda x)}{\lambda}
  = \| F^{\infty}(x) \| + \sum_{i = 1}^n \max\{ G_i^{\infty}(x), 0 \}
\] 
is correctly defined and $S(\varphi^{\infty}) = S(f^{\infty})$. Hence applying the results of the previous section to
the function $f$ and taking into account the fact that the norms $\| \cdot \|$ and $\| \cdot \|_{\infty}$ are
equivalent, one obtains the required result.
\end{proof}

\bibliographystyle{abbrv}  
\bibliography{PWAff_ErrorBounds_bibl}

\begin{thebibliography}{10}

\bibitem{AddebTroitsky}
S.~Adeeb and V.~G. Troitsky.
\newblock Locally piecewise affine functions and their order structure.
\newblock {\em Positivity}, 21:213--221, 2017.

\bibitem{AliprantisHarrisTourky}
C.~Aliprantis, D.~Harris, and R.~Tourky.
\newblock Continuous piecewise linear functions.
\newblock {\em Macroecon. Dyn.}, 10:77--99, 2006.

\bibitem{Angelov}
T.~Angelov.
\newblock Representation of piecewise affine functions as a difference of
  polyhedral.
\newblock {\em Vestnik St. Petersburg University. Ser. 10. Prikl. Mat. Inform.
  Prots. Upr.}, 1:4--18, 2016.
\newblock [In Russian]. Available at:
  https://cyberleninka.ru/article/n/predstavlenie-kusochno-affinnyh-funktsiy-v-vide-raznosti-poliedralnyh/viewer.

\bibitem{Aze}
D.~Az\'{e}.
\newblock A survey on error bounds for lower semicontinuous functions.
\newblock {\em ESAIM: Proc.}, 13:1--17, 2003.

\bibitem{BolteNguyen}
J.~Bolte, T.~P. Nguyen, J.~Peypouquet, and B.~W. Suter.
\newblock From error bounds to the complexity of first-order descent methods
  for convex functions.
\newblock {\em Math. Program.}, 165:471--507, 2017.

\bibitem{BonnansShapiro}
J.~F. Bonnans and A.~Shapiro.
\newblock {\em Perturbation Analysis of Optimization Problems}.
\newblock Springer, New York, 2000.

\bibitem{CuongKruger}
H.~D. Cuong and A.~Y. Kruger.
\newblock Error bounds revisited.
\newblock {\em Optim.}, 71:1021--1053, 2022.

\bibitem{Deng}
S.~Deng.
\newblock Global error bounds for convex inequality systems in {B}anach spaces.
\newblock {\em SIAM J. Control Optim.}, 36:1240--1249, 1998.

\bibitem{Dolgopolik_OMS}
M.~V. Dolgopolik.
\newblock The method of codifferential descent for convex and global piecewise
  affine optimization.
\newblock {\em Optim. Methods Softw.}, 35:1191--1222, 2020.

\bibitem{FabianHenrion}
M.~J. Fabian, R.~Henrion, A.~Y. Kruger, and J.~V. Outrata.
\newblock Error bounds: necessary and sufficient conditions.
\newblock {\em Set-Valued Anal.}, 18:121--149, 2010.

\bibitem{Gorokhovik}
V.~V. Gorkohivik.
\newblock Geometrical and analytical characteristic properties of piecewise
  affine mappings.
\newblock {\em arXiv: 1111.1389}, pages 1--12, 2011.

\bibitem{GorokhovikZorko}
V.~V. Gorokhovik and O.~I. Zorko.
\newblock Piecewise affine functions and polyhedral sets.
\newblock {\em Optim.}, 31:209--221, 1994.

\bibitem{Gowda1996}
M.~S. Gowda.
\newblock An analysis of zero set and global error bound properties of a
  piecewise affine function via its recession function.
\newblock {\em SIAM J. Matrix Anal. Appl.}, 17:594--–609, 1996.

\bibitem{Lemarechal}
J.-B. Hiriart-{U}rruty and C.~Lemar\'{e}chal.
\newblock {\em Convex Analysis and Minimization Algorithms {I}. Fundamentals}.
\newblock Springer-Verlag, Berlin, Heidelberg, 1993.

\bibitem{Hoffman}
A.~J. Hoffman.
\newblock On approximate solutions of systems of linear inequalities.
\newblock {\em J. Research of the National Bureau of Standards}, 49:263--265,
  1952.

\bibitem{KripfganzSchulze}
A.~Kripfganz and R.~Schulze.
\newblock Piecewise affine functions as a difference of two convex functions.
\newblock {\em Optim.}, 18:23--29, 1987.

\bibitem{Kruger}
A.~Y. Kruger.
\newblock Error bounds and metric subregularity.
\newblock {\em Optim.}, 64:49--79, 2015.

\bibitem{KuhnLowen}
D.~Kuhn and R.~L\"{o}wen.
\newblock Piecewise affine bijections of ${R}^n$, and the equation ${S}x^+ -
  \mathcal{T} x^- = y$.
\newblock {\em Linear Algebra Appl.}, 96:109--129, 1987.

\bibitem{LeThiDing}
H.~A. {Le Thi}, T.~{Pham Dinh}, and H.~V. Ngai.
\newblock Exact penalty and error bounds in {D}{C} programming.
\newblock {\em J. Glob. Optim.}, 52:509--535, 2012.

\bibitem{Li}
G.~Li.
\newblock Global error bounds for piecewise convex polynomials.
\newblock {\em Math Program.}, 137:37--64, 2013.

\bibitem{LiSinger}
W.~Li and I.~Singer.
\newblock Global error bounds for convex multifunctions and applications.
\newblock {\em Math. Oper. Res.}, 23:443--462, 1998.

\bibitem{Ovchinnikov}
S.~Ovchinnikov.
\newblock Max-min representation of piecewise linear functions.
\newblock {\em Beitr\"{a}ge zur Algebra und Geometrie}, 43:297--302, 2002.

\bibitem{Pang}
J.~Pang.
\newblock Error bounds in mathematical programming.
\newblock {\em Math. Program.}, 79:299--332, 1997.

\bibitem{Radons}
M.~Radons.
\newblock A note on surjectivity of piecewise affine mappings.
\newblock {\em Optim. Letters}, 13:439--443, 2019.

\bibitem{RheinboldtVandergraft}
W.~C. Rheinboldt and J.~S. Vandergraft.
\newblock On piecewise affine mappings in ${R}^n$.
\newblock {\em SIAM J. Appl. Math.}, 29:680--689, 1975.

\bibitem{Robinson}
S.~M. Robinson.
\newblock Some continuity properties of polyhedral multifunctions.
\newblock In H.~K\"{o}nig, B.~Korte, and K.~Ritter, editors, {\em Mathematical
  Programming at Oberwolfach. Mathematical Programming Studies, vol. 14}, pages
  206--214. Springer, Berlin, Heidelberg, 1981.

\bibitem{Rockafellar}
R.~T. Rockafellar.
\newblock {\em Convex Analysis}.
\newblock Princeton University Press, Princeton, 1970.

\bibitem{SchluterDarup}
N.~Schl\"{u}ter and M.~S. Darup.
\newblock Novel convex decomposition of piecewise affine functions.
\newblock {\em arXiv: 2108.03950}, pages 1--10, 2021.

\bibitem{Vui}
H.~H. Vui.
\newblock Global holderian error bound for nondegenerate polynomials.
\newblock {\em SIAM J. Optim.}, pages 917--933, 2013.

\bibitem{WangPang}
T.~Wang and J.~Pang.
\newblock Global error bounds for convex quadratic inequality systems.
\newblock {\em Optim.}, 31:1--12, 1994.

\end{thebibliography}

\end{document}